\documentclass{amsart}
\usepackage{kotex}
\usepackage{amsmath}
\usepackage{amssymb}
\usepackage{amscd}
\usepackage[all]{xy}
\usepackage{color}
\usepackage{mathrsfs}
\usepackage{tikz-cd}
\mathchardef\mhyphen="2D

\newtheorem{theorem}{Theorem}[section]
\newtheorem{lemma}[theorem]{Lemma}
\newtheorem{corollary}[theorem]{Corollary}
\newtheorem{proposition}[theorem]{Proposition}

\newtheorem{example}[theorem]{Example}
\newtheorem{remark}[theorem]{Remark}

\begin{document}
\title
[The Anderson ring]
{A special subring of the Nagata ring and\\ the Serre's conjecture ring}

\author [H. Baek] {Hyungtae Baek}
\address{(Baek) School of Mathematics,
Kyungpook National University, Daegu 41566,
Republic of Korea}
\email{htbaek5@gmail.com}

\author [J. W. Lim] {Jung Wook Lim}
\address{(Lim) Department of Mathematics,
College of Natural Sciences,
Kyungpook National University, Daegu 41566,
Republic of Korea}
\email{jwlim@knu.ac.kr}

\thanks{Words and phrases:
Anderson ring, Nagata ring, Serre's conjecture ring, von Neumann regular ring}

\thanks{$2020$ Mathematics Subject Classification: 13A15, 13B25, 13B30}



\begin{abstract}
Many ring theorists researched various properties of Nagata rings and Serre's conjecture rings.
In this paper, we introduce a subring (refer to the Anderson ring) of both the Nagata ring and the Serre's conjecture ring (up to isomorphism), and investigate properties of the Anderson ring.
Additionally, we compare the properties of the Anderson ring with those of the Nagata ring and the Serre's conjecture ring.
\end{abstract}

\maketitle

\section{Introduction}

Throughout this paper, $R$ is a commutative ring with identity and
$R[X]$ is the polynomial ring over $R$.
For the sake of clarity, we use $D$
instead of $R$ when $R$ is an integral domain.
Additionally, ${\rm Spec}(R)$ is the set of prime ideals of $R$,
and ${\rm Max}(R)$ is the set of maximal ideals of $R$.

Consider a field $K$ and
let $\alpha \in K$.
Recall that the {\it localization} at $\alpha$,
denoted by $K[X]_{M_{\alpha}}$,
is the set $\{\frac{f}{g} \,|\, f,g \in K[X]
\text{ and } g(\alpha) \neq 0\}$.
When $\alpha = 0$,
we obtain the ring
$ K[X]_{M_0} = K[X]_{(X)} =
\{\frac{f}{g} \,|\, f,g \in K[X] \text{ and } g(0) \neq 0 \}$.
Observe that the complement of the ideal $(X)$
is the set of polynomials over $K$
whose constant term is a unit in $K$.
To generalize this,
consider the set
$\overline{A} :=
\{f \in R[X] \,|\, f(0) \text{ is a unit in $R$}\}$
which is the saturation of
$A := \{f \in R[X] \,|\, f(0) = 1\}$.
Also, it is clear that $A$ is a multiplicative subset of $R[X]$,
so we obtain the quotient ring $R[X]_A$ of $R[X]$ by $A$.
In \cite{anderson 1985},
the authors mentioned the set $A$,
so we refer to the ring $R[X]_A$
as the {\it Anderson ring} of $R$.
In this paper, we examine some properties of the Anderson ring.

Consider the following multiplicative subsets of $R[X]$:
\begin{center}
$N = \{f \in R[X] \,|\, c(f) = R\}$ and
$U = \{f \in R[X] \,|\, f \text{ is monic}\}$,
\end{center}
where $c(f)$ is the ideal of $R$ generated by
the coefficients of $f$.
The quotient ring $R[X]_N$ is called the {\it Nagata ring} of $R$, and
the quotient ring $R[X]_U$ is called the {\it Serre's conjecture ring} of $R$. (Some authors denote $R[X]_N$ by
$R(X)$, and $R[X]_U$ by $R\left<X\right>$.)
In 1936, Krull constructed Nagata rings \cite{krull 1936}, and then
it was studied by Nagata \cite{nagata 1956, nagata 1962}.
In 1955, Serre posed the question:
`It is not known whether there exist projective $k[X_1,\dots,X_n]$-modules of finite type which are not free,
where $k$ is a field' \cite{serre 1955}.
In 1976, Quillen suggested a solution to this question: if $D$ is a principal ideal domain,
then every finitely generated projective $D[X_1,\dots,X_n]$-module is free \cite{quillen 1976}.
To obtain the above answer,
the author constructed the Serre's conjecture rings.
The reader can refer to \cite{anderson 1985, brewer 1980, chang 2005, huckaba 1988, jarrar 2018, kang 1989, lim 2014, lucas 2020, wang book}
for the Nagata rings and the Serre's conjecture rings.

Let $\widetilde{U}$ be the set of polynomials
whose coefficient of the lowest degree term is $1$.
Then it is clear that $\widetilde{U}$ is a multiplicative subset of $R[X]$
containing the set $A$.
Additionally, the map
$R[X]_U \to R[X]_{\widetilde{U}}$ given by
$X \mapsto X^{-1}$ is an isomorphism.
Hence we obtain the facts that
$R[X]_A$ is a subring of $R[X]_U$ in the isomorphic sense,
and $R[X]_A$ is a subring of $R[X]_N$.
More precisely, it is easy to check that
$(R[X]_A)[\frac{1}{X}] = R[X]_{\widetilde{U}}$,
which means that Serre's conjecture ring is an extension of the Anderson ring.

Note that $R[X]_N$ and $R[X]_U$ are faithfully flat $R$-modules,
so $R[X]_N$ and $R[X]_U$ share many ideal and ring-theoretic properties with $R$.
Similarly, it is easy to show that $R[X]_A$ is also a faithfully flat $R$-module,
so we can expect that
the Anderson ring and their base rings share many
ideal and ring-theoretic properties.
In this paper, we examine some ideal and ring-theoretic properties shared by the Anderson ring and the base rings.

This paper consists of four sections including introduction.
In Section \ref{sec 2},
we investigate the maximal spectrum of the Anderson ring.
We show that ${\rm Max}(R[X]_A) = \{(M+XR[X])_A \,|\, M \in {\rm Max}(R)\}$ (Theorem \ref{maximal cor}).
After examining the maximal spectrum of the Anderson ring,
we investigate some properties of the Anderson ring
related to maximal spectrum.
We show that $\dim(R[X]_A) = \dim(R[X])$ (Proposition \ref{dimension}), and
we examine local properties of the Anderson ring.
In Section \ref{sec 3},
we investigate the Anderson ring over von Neumann regular rings.
We first show that
$R$ is both a von Neumann regular ring and a PIR if and only if
$R[X]_A$ is a one-dimensional PIR (Theorem \ref{PIR 2}).
Also, we investigate the condition of $R$ under which
the Anderson ring become Pr\"ufer-like rings
({\it e.g.}, semi-hereditary ring, arithmetical ring, Gaussian ring, etc)
(Theorem \ref{von, Prufer-like}).
In Section \ref{sec 4},
we examine star-operations on the Anderson ring.
More precisely, we investigate the $w$-maximal spectrum of the Anderson ring;
we show that
$w\mhyphen{\rm Max}(D[X]_A) =
\{MD[X]_A \,|\, w\mhyphen{\rm Max}(D)\}
\cup \{\mathfrak{p}D[X]_A \,|\,
\text{$\mathfrak{p} \in w$-Max$(D[X])$ is an upper to zero in $D[X]$ disjoint from $A$} \}$
(Theorem \ref{maximal w-ideal}).
After examining the above fact,
we investigate some properties of the Anderson ring
which are related to the $w$-maximal spectrum of the Anderson ring.
We prove that $D$ has finite $w$-character if and only if
$D[X]_A$ has finite $w$-character (Proposition \ref{w-finite character}),
and we also examine $w$-local properties of the Anderson ring.

We should notice that some of the results
(Theorems \ref{maximal cor} and \ref{von, Prufer-like}, Propositions \ref{dimension} and \ref{Hilbert ring}, and Remark \ref{dimension remark}(2)(ii))
also appeared in \cite{a}.
However, such results were already contained in \cite{baek},
which is the master thesis of the first author,
prior to the publication of \cite{a}.
Also, we include the proof of such results for the convenience of the readers.

\section{Maximal ideals of $R[X]_A$}\label{sec 2}

Let $R$ be a commutative ring with identity.
In this section, we investigate maximal ideals of $R[X]_A$,
and then we examine some properties of $R[X]_A$
that follow from the properties of its maximal ideals.
A well-known fact is that there is a one-to-one correspondence
between the maximal ideals of $R$ and
the maximal ideals of $R[X]_N$.
In fact, ${\rm Max}(R[X]_N) = \{MR[X]_N \,|\, M \in {\rm Max}(R)\}$
\cite[Proposition 33.1(3)]{gilmer book}.
In \cite{riche 1980}, Le Riche showed that
the extension of a maximal ideal of $R$ to $R[X]_U$ is also a maximal ideal of $R[X]_U$,
but there is no one-to-one correspondence between the maximal ideals of $R$ and
the maximal ideals of $R[X]_U$ when $R$ is a one-dimensional integral domain
\cite[Lemma 3.2]{lucas 2020}.
The next result shows that the maximal spectrum of $R[X]_A$ can be characterized.

\begin{theorem}\label{maximal cor}
Let $R$ be a commutative ring with identity.
Then the following assertions hold.
\begin{enumerate}
\item[(1)]
There is a one-to-one correspondence between
the minimal prime ideals of $R$ and the minimal prime ideals of $R[X]_A$.
In fact, every minimal prime ideal of $R[X]_A$ is of the form $PR[X]_A$
for some minimal prime ideal $P$ of $R$.
\item[(2)]
There is a one-to-one correspondence between
the maximal ideals of $R$ and the maximal ideals of $R[X]_A$.
In fact, ${\rm Max}(R[X]_A) = \{(M+XR[X])_A \,|\, M \in {\rm Max}(R)\}$.
\end{enumerate}
\end{theorem}

\begin{proof}
(1) Let $P$ be a minimal prime ideal of $R$.
Then $PR[X]$ is a minimal prime ideal of $R[X]$ disjoint from $A$,
so $PR[X]_A$ is a minimal prime ideal of $R[X]_A$.
Now, suppose that $\mathfrak{p}$ be a minimal prime ideal of $R[X]_A$.
Then there exists a minimal prime ideal $\mathcal{P}$ of $R[X]$ such that $\mathfrak{p} = \mathcal{P}R[X]_A$.
Note that every minimal prime ideal of $R[X]$ is the extension of a minimal prime ideal of $R$.
Thus $\mathfrak{p}= PR[X]_A$ for some minimal prime ideal $P$ of $R$.

(2) Let $M$ be a maximal ideal of $R$.
As $M + XR[X]$ is a maximal ideal of $R[X]$ disjoint from $A$,
$(M + XR[X])_A$ is a maximal ideal of $R[X]_A$.
Now, suppose that $\mathfrak{m}$ is a maximal ideal of $R[X]_A$.
Then there exists a prime ideal $\mathfrak{p}$ of $R[X]$ disjoint from $A$
such that $\mathfrak{m} = \mathfrak{p}R[X]_A$.
Also, it is easy to check that $I := \{f(0) \,|\, f \in \mathfrak{p} \}$
is a proper ideal of $R$.
Hence there is a maximal ideal $M$ of $R$ which containing $I$,
which means that $\mathfrak{p} \subseteq M+ XR[X]$.
As $\mathfrak{p}R[X]_A$ is a maximal ideal of $R[X]_A$,
$\mathfrak{m} = \mathfrak{p}R[X]_A = (M+XR[X])_A$.
Consequently,
every maximal ideal of $R[X]_A$ is of the form $(M+XR[X])_A$
for some maximal ideal $M$ of $R$.
\end{proof}

\subsection{Results from Theorem \ref{maximal cor}}

By Theorem \ref{maximal cor},
we can derive several results.
In this subsection, we discuss some of the results from Theorem \ref{maximal cor}.

\begin{corollary}\label{quasi-local}
Let $R$ be a commutative ring with identity. Then
$R$ is a semi-quasi-local ring if and only if $R[X]_A$ is a semi-quasi-local ring.
In particular, $R$ is a quasi-local ring if and only if $R[X]_A$ is a quasi-local ring.
\end{corollary}

Recall that a commutative ring $R$ with identity is said to have {\it finite character} if
every nonzero nonunit element belongs to only finitely many maximal ideals of $R$.

\begin{proposition}\label{finite character}
Let $R$ be a commutative ring with identity.
Then $R[X]_A$ is of finite character if and only if $R$ is a semi-quasi-local ring.
\end{proposition}

\begin{proof}
Suppose that $R[X]_A$ is of finite character.
Note that $X$ is a nonzero nonunit of $R[X]_A$ and
$X$ is contained in $(M+XR[X])_A$ for all $M \in {\rm Max}(R)$,
which means that ${\rm Max}(R)$ is a finite set by Theorem \ref{maximal cor}(2).
Thus $R$ is a semi-quasi-local ring.
The converse directly follows from Corollary \ref{quasi-local}.
\end{proof}

Now, we investigate the Krull dimension of the Anderson ring.
The next result shows that the Krull dimension of the Anderson ring is
always equal to the Krull dimension of the polynomial rings,
and this result is very useful in this paper.

\begin{proposition}\label{dimension}
Let $R$ be a commutative ring with identity.
If $R$ is finite dimensional, then $\dim(R[X]) = \dim(R[X]_A)$,
and hence $n+1 \leq \dim(R[X]_A) \leq 2n+1$ if $\dim(R) = n$.
\end{proposition}

\begin{proof}
Note that there exists a maximal ideal $M$ of $R$ such that
$\dim(R[X]) = {\rm rank}(M+XR[X])$.
As $(M+XR[X])_A$ is a maximal ideal of $R[X]_A$,
we obtain
\begin{eqnarray*}
\dim(R[X]_A) &=& {\rm rank}((M + XR[X])_A)\\
&=&{\rm rank}(M+XR[X]) \\
&=& \dim(R[X]).
\end{eqnarray*}
The remaining argument directly follows from \cite[Theorem 2]{seidenberg 1953}.
\end{proof}

\begin{remark}\label{dimension remark}
{\rm
Let $R$ be a commutative ring with identity.

(1) As $R[X]$ is never zero-dimensional,
$R[X]_A$ is also never zero-dimensional by Proposition \ref{dimension}.

(2) Recall that $n+1 \leq \dim(R[X]) \leq 2n+1$ if $\dim(R) = n$
\cite[Theorem 2]{seidenberg 1953}; and
if $R$ is an $n$-dimensional Noetherian ring,
then $\dim(R[X]) = n+1$ \cite[Theorem 9]{seidenberg 1953}.
Using the above fact and Proposition \ref{dimension},
we obtain the following facts.
\begin{enumerate}
\item[(i)]
$R$ is zero-dimensional if and only if $R[X]_A$ is one-dimensional.
In this case, ${\rm Spec}(R[X]_A) = \{MR[X]_A \,|\, M \in {\rm Max}(R)\} \cup
\{(M+XR[X])_A \,|\, M \in {\rm Max}(R)\}$ by Theorem \ref{maximal cor}.
In particular,
if $R$ is an integral domain,
then $R$ is a field if and only if
$R[X]_A$ is one-dimensional.
\item[(ii)]
If $R$ is a Noetherian ring,
then $\dim(R[X]_A) = \dim(R) + 1$.
\end{enumerate}

(3) Remind that $R$ is an Artinian ring if and only if
$R$ is a zero-dimensional Noetherian ring \cite[Theorem 8.5]{atiyah book};
and $D$ is a Dedekind domain if and only if
$D$ is a one-dimensional integrally closed Noetherian domain \cite[Theorem 5.2.15]{wang book}.
Hence $R[X]_A$ is never an Artinain ring by (1),
and $R$ is a field if and only if $R[X]_A$ is a Dedekind domain by (2).

(4) Note that $R[X]_N$ and $R[X]_U$ coincide if and only if
$\dim(R) = 0$ \cite[Theorem 17.11]{huckaba 1988}.
Since $\dim(R[X]_N) = \dim(R[X])-1 = \dim(R[X]_U)$
\cite[Theorem 17.3 and Corollary 17.4]{riche 1980},
$R[X]_A$ never coincides with $R[X]_N$ and $R[X]_U$ by Proposition \ref{dimension}.
}
\end{remark}

Throughout this paper,
we denote
$A_P = \{f \in R_P[X] \,|\, f(0) \text{ is a unit in $R_P$}\}$
for any prime ideal $P$ of $R$.
As there is a one-to-one correspondence between
the maximal ideals of $R$ and the maximal ideals of $R[X]_A$,
we derive some local properties of $R[X]_A$.
The next result is a useful tool
for investigating the local properties of the Anderson ring.

\begin{lemma}\label{R_P[X]_A_P}
Let $R$ be a commutative ring with identity.
Then the following statements hold.
\begin{enumerate}
\item[(1)]
The saturation of $A$ is the set of polynomials whose constant term is a unit in $R$.
\item[(2)]
If $\overline{A}$ is the saturation of $A$,
then $\overline{A} = R[X] \setminus\bigcup_{M \in {\rm Max}(R)}(M + XR[X])$.
\item[(3)]
If $R$ is quasi-local with maximal ideal $M$, then $R[X]_A = R[X]_{M+XR[X]}$.
\item[(4)]
For a prime ideal $P$ of $R$,
$R_P[X]_{A_P} = (R[X]_A)_{(P+XR[X])_A}$.
\end{enumerate}
\end{lemma}

\begin{proof}
(1) This result is obvious.

(2) Let $f \in \overline{A}$.
Then $f(0)$ is a unit in $R$
by the assertion (1).
This implies that $f(0) \in R \setminus \bigcup_{M \in {\rm Max}(R)}M$,
and hence $f \in R[X] \setminus \bigcup_{M \in {\rm Max}(R)}(M+XR[X])$.
For the reverse containment,
let $f \in R[X] \setminus \bigcup_{M \in {\rm Max}(R)}(M+XR[X])$.
Then $f(0) \in R \setminus \bigcup_{M \in {\rm Max}(R)}M$.
This follows that $f(0)$ is a unit in $R$.
Thus $f \in \overline{A}$.

(3) The result follows directly from the assertion (2).

(4) Note that $R_P$ is a quasi-local ring with maximal ideal $PR_P$.
Thus we have
\begin{eqnarray*}
R_P[X]_{A_P} &=& R_P[X]_{PR_P+XR_P[X]}\\
&=& R[X]_{P+XR[X]}\\
&=& (R[X]_A)_{(P+XR[X])_A},
\end{eqnarray*}
where the first equality follows directly from the assertion (3).
\end{proof}

Let $R$ be a commutative ring with identity.
Recall that $R$ is a {\it locally Noetherian ring} if
$R_M$ is a Noetherian ring for all $M \in {\rm Max}(R)$.
The following result is a representative local property of the Anderson ring.

\begin{proposition}\label{locally Noetherian}
Let $R$ be a commutative ring with identity.
Then the following statements are equivalent.
\begin{enumerate}
\item[\rm(1)]
$R$ is a locally Noetherian ring.
\item[\rm(2)]
$R[X]_A$ is a locally Noetherian ring.
\end{enumerate}
\end{proposition}

\begin{proof}
As $(R[X]_A)_N = R[X]_N$ and
$R$ is a Noetherian ring if and only if $R[X]$ is a Noetherian ring
if and only if $R[X]_N$ is a Noetherian ring,
we obtain the fact that
$R$ is a Noetherian ring if and only if $R[X]_A$ is a Noetherian ring.

(1) $\Rightarrow$ (2)
Let $\mathfrak{m}$ be a maximal ideal of $R[X]_A$.
Then by Theorem \ref{maximal cor}(2), there exists a maximal ideal $M$ of $R$
such that $\mathfrak{m} = (M + XR[X])_A$.
Since $R_M$ is Noetherian,
$R_M[X]_{A_M}$ is also Noetherian.
This implies that $(R[X]_A)_{\mathfrak{m}}=(R[X]_A)_{(M + XR[X])_A}$
is a Noetherian ring by Lemma \ref{R_P[X]_A_P}(4).
Thus $R[X]_A$ is a locally Noetherian ring.

(2) $\Rightarrow$ (1)
Let $M$ be a maximal ideal of $R$.
Then $(M + XR[X])_A$ is a maximal ideal of $R[X]_A$ by Theorem \ref{maximal cor}(2).
Hence $(R[X]_A)_{(M + XR[X])_A}$ is Noetherian,
which implies that $R_M[X]_{A_M}$ is Noetherian by Lemma \ref{R_P[X]_A_P}(4).
Thus $R_M$ is a Noetherian ring.
Consequently, $R$ is a locally Noetherian ring.
\end{proof}

Note that every locally Noetherian ring of finite character is a Noetherian ring
\cite[Section 7, Exercise 9]{atiyah book}.
Thus by Propositions \ref{finite character} and \ref{locally Noetherian}, we obtain

\begin{corollary}\label{locally Noetherian with finite character}
Let $R$ be a commutative ring with identity.
If $R[X]_A$ is a locally Noetherian ring with finite character,
then $R$ is a semi-quasi-local Noetherian ring.
\end{corollary}

Similarly to Proposition \ref{locally Noetherian},
we can obtain a lot of local properties of the Anderson ring.

\begin{remark}
{\rm
Let $R$ be a commutative ring with identity.
Let (P) be a property which satisfies that
$R$ has a property (P) if and only if $R[X]_A$ has a property (P).
Then we obtain that
$R_M$ has a property (P) for all $M \in {\rm Max}(R)$ if and only if
$(R[X]_A)_{\mathfrak{m}}$ has a property (P) for all $\mathfrak{m} \in {\rm Max}(R[X]_A)$.
}
\end{remark}

Let $D$ be an integral domain, $M$ a maximal ideal of $D$ and
let $N_M = \{f \in D_M[X] \,|\, c(f) = D_M\}$.
Recall that $D[X]_N = \bigcap_{M \in {\rm Max}(D)}D_M[X]_{N_M}$
\cite[Proposition 2.9]{kang 1989}.
Similarly, we obtain the following result,
which shows that
$D[X]_A$ can be expressed as the intersection of
quasi-local Anderson rings.

\begin{proposition}
Let $D$ be an integral domain.
Then $D[X]_A = \bigcap_{M \in {\rm Max}(D)}D_M[X]_{A_M}$.
\end{proposition}

\begin{proof}
By Theorem \ref{maximal cor}(2) and Lemma \ref{R_P[X]_A_P}(4),
we obtain
\begin{eqnarray*}
D[X]_A &=& \bigcap_{M \in {\rm Max}(D)}(D[X]_A)_{(M + XD[X])_A}\\
&=& \bigcap_{M \in {\rm Max}(D)} D_M[X]_{A_M},
\end{eqnarray*}
where the first equality follows from
\cite[Theorem 4.10(3)]{gilmer book}.
\end{proof}

Let $R$ be a commutative ring with identity and
let $D$ be an integral domain with quotient field $K$. Recall that
\begin{itemize}
\item $D$ is a {\it G-domain} if $K$ is a finitely generated ring over $D$,
\item a prime ideal $P$ of $R$ is  a {\it G-ideal} if $R/P$ is a G-domain, and
\item $R$ is a {\it Hilbert ring} if every G-ideal is maximal.
\end{itemize}
In \cite{brewer 1980}, Brewer and Heinzer showed that if $R$ is a Noetherian ring,
then $R[X]_U$ is a Hilbert ring,
and in \cite{anderson 1985}, the authors showed that
$R[X]_N$ is a Hilbert ring if and only if
$R$ is a Hilbert ring and
${\rm Spec}(R[X]_N) = \{PR[X]_N \,|\, P \in {\rm Spec}(R)\}$ if and only if
$R$ is a Hilbert ring and $\overline{R/P}$ is
a Pr\"ufer domain for any minimal prime ideal $P$ of $R$,
where $\overline{R/P}$ is the integral closure of $R/P$.
We conclude this section with the following result, which provides a useful tool for constructing many examples of rings that are not Hilbert rings.

\begin{proposition}\label{Hilbert ring}
Let $R$ be a commutative ring with identity.
Then $R[X]_A$ is never a Hilbert ring.
\end{proposition}

\begin{proof}
Suppose to the contrary that $R[X]_A$ is a Hilbert ring.
Let $P$ be a prime ideal of $R$.
Then $PR[X]_A$ is a prime ideal of $R[X]_A$.
Hence $PR[X]_A$ can be expressed
as an intersection of the maximal ideals of $R[X]_A$
properly containing $PR[X]_A$
\cite[Theorem 31.8]{gilmer book}.
Let $\{M_{\alpha} \,|\, \alpha \in \Lambda \}$
be the set of maximal ideals of $R$ properly containing $P$.
Then by Theorem \ref{maximal cor}(2),
$PR[X]_A = \bigcap_{\alpha \in \Lambda}( (M_{\alpha}+XR[X])_A)$.
This follows that $X \in PR[X]_A$.
This contradicts to the fact that $1 \notin P$.
Thus $R[X]_A$ is never a Hilbert ring.
\end{proof}

\section{$R[X]_A$ over von Neumann regular rings}\label{sec 3}

Let $R$ be a commutative ring with identity.
Recall that $R$ is a {\it von Neumann regular ring} if
for any $a \in R$,
there exists an element $b \in R$ such that
$a^2b = a$.
A useful fact is that
$R$ is a von Neumann regular ring if and only if
$R$ is a zero-dimensional reduced ring,
which is also equivalent to
$R_P$ being a field for all $P \in {\rm Spec}(R)$
\cite[Theorem 3.6.16]{wang book}
(recall that a reduced ring is a ring that has no nonzero nilpotent elements).
In this section, we investigate the Anderson ring whose
base ring is a von Neumann regular ring.

First, we examine the condition on $R$ under which
$R[X]_A$ becomes a principal ideal ring (for short, PIR).
Note that if $R$ is a PIR, then $\dim(R) \leq 1$,
so $R[X]_A$ is one-dimensional
whenever $R[X]_A$ is a PIR
by Remark \ref{dimension remark}(1).
In \cite{riche 1980},
Le Riche showed that $R$ is a PIR if and only if
$R[X]_U$ is a PIR, and
in \cite{anderson 1985},
the authors showed that
$R$ is a PIR if and only if $R[X]_N$ is a ZPI ring.
Base on the above facts, we can naturally consider the question `When is $R[X]_A$ a PIR?'.
The following result is the answer to the above question which is the first main result of this section.

\begin{theorem}\label{PIR 2}
Let $R$ be a commutative ring with identity.
Then the following are equivalent.
\begin{enumerate}
\item[\rm(1)]
$R$ is both a von Neumann regular ring and a PIR.
\item[\rm(2)]
$R[X]$ is a one-dimensional PIR.
\item[\rm(3)]
$R[X]_A$ is a one-dimensional PIR.
\end{enumerate}
\end{theorem}

\begin{proof}
If $R$ is an integral domain,
then the result holds obviously.
Hence suppose that $R$ is not an integral domain.

(1) $\Rightarrow$ (2)
As $R[X]$ is one-dimensional \cite[Theorem 2]{seidenberg 1953},
it is sufficient to show that $R[X]$ is a PIR.
Note that every PIR can be expressed as a direct sum of
PIDs and special PIRs \cite[Chapter IV, Theorem 33]{zariski book}
(recall that a special PIR is a quasi-local ring which has nonzero nilpotent maximal ideal).
Since every special PIR contains nonzero nilpotent and $R$ is reduced,
we obtain that $R = D_1\oplus \cdots \oplus D_m$,
where $D_1,\dots,D_m$ are PIDs.
Since an ideal $I$ of $R$ to be a prime ideal,
all components except exactly one are the entire rings,
we obtain that
the prime ideals of $R$ are correspondence to the prime ideals of some $D_i$.
This follows that if there exists $1 \leq i \leq m$ such that $D_i$ is not a field,
then $R$ is one-dimensional.
As $R$ is zero-dimensional, $R$ can be expressed as a direct sum of fields.
This follows that $R[X]$ can be expressed as a direct sum of PIDs,
and thus $R[X]$ is a one-dimensional PIR.

(2) $\Rightarrow$ (3) The result is obvious.

(3) $\Rightarrow$ (1)
As $R$ is a zero-dimensional PIR,
it is sufficient to show that $R$ is a reduced ring \cite[Theorem 3.6.16]{wang book}.
Suppose to the contrary that there exists
a nonzero element $a \in R \setminus \{0\}$ such that
$a^n = 0$ for some $n \geq 2$.
Let $I = (a)$.
As $(I + XR[X])_A$ is principal,
there exists $f := \sum_{i=1}^{s} a_iX^i \in I+XR[X]$
such that $(I + XR[X])_A = fR[X]_A$.
This follows that $a_0 \in I$, so $a_0^n = 0$.
Let $k$ be the smallest positive integer satisfying $a_0^k = 0$
and suppose to the contrary that $k \geq 2$.
Since $X \in fR[X]_A$,
there exist polynomials $g:= \sum_{i=1}^{t} b_iX^i \in R[X]$ and $h \in A$
such that $X = f \frac{g}{h}$.
Hence $a_0b_0 = 0$ and $a_0b_1+a_1b_0 = 1$.
Multiplying the second equation by $a_0$,
we obtain the equation $a_0^2b_1 = a_0$.
This equation implies that $a_0^{k-1} = a_0^kb_1 = 0$.
This contradicts to the minimality of $k$.
This implies that $a_0 = 0$.
Now, consider the element $\frac{a}{X+1} \in (I+XR[X])_A = fR[X]_A$.
Then $\frac{a}{X+1} = f \frac{g_1}{h_1}$ for some $g_1 \in R[X]$ and $h_1 \in A$.
This implies the equation $a = ah_1(0) = a_0g_1(0) = 0$,
so $a = 0$, which contradicts our choice of $a$.
Thus $R$ is a zero-dimensional reduced PIR.
\end{proof}

Note that $n$ is a square-free positive integer if and only if
$\mathbb{Z}_n$ is a von Neumann regular ring.
Also, it is clear that $\mathbb{Z}_n$ is a PIR.
Hence we can obtain the following result from Theorem \ref{PIR 2}.

\begin{corollary}
Let $n$ be a positive integer and
$R = \mathbb{Z}_n$.
Then $R[X]_A$ is a PIR if and only if
$n$ is square free.
\end{corollary}

Let $R$ be a commutative ring with identity and
let $M$ be an $R$-module.
If $M$ has a flat resolution
$
\begin{tikzcd}[
   column sep=0.80em,
  ]
0\arrow[r]
  &F_n
  \arrow[r]
  &F_{n-1}
  \arrow[r]
  &\cdots
  \arrow[r]
  &F_1
  \arrow[r]
  &F_0
  \arrow[r]
  &M,
\end{tikzcd}
$
then we say the {\it flat dimension} of $M$ is at most $n$.
If $n$ is the smallest such integer,
then we define the flat dimension of $M$ is $n$,
and denoted by ${\rm fd}_R(M) = n$.
If there is no finite flat resolution of $M$,
then define ${\rm fd}_R(M) = \infty$.
Also, the {\it weak global dimension} of $R$, and
denoted by w.gl.$\dim(R)$,
is defined by
w.gl.$\dim(R) = \sup\{{\rm fd}_R(M) \,|\, M \text{ is an $R$-module}\}$.
Recall that $R$ is a {\it Pr\"ufer domain}
if $R$ is an integral domain and w.gl.$\dim(R) \leq 1$.
Pr\"ufer domains are characterized by many equivalent conditions.
Many of these conditions have been extended to the case of rings with zero-divisors
and gave rise to at least six classes of Pr\"ufer-like rings, namely:
\begin{enumerate}
\item[(1)]
$R$ is a {\it semi-hereditary ring}
if every finitely generated ideal of $R$ is projective.
\item[(2)]
$R$ is an {\it arithmetical ring}
if every finitely generated ideal of $R$ is locally principal.
\item[(3)]
$R$ is a {\it Gaussian ring}
if $c(fg) = c(f)c(g)$ for any $f,g \in R[X]$.
\item[(4)]
$R$ is a {\it locally Pr\"ufer ring}
if $R_P$ is a Pr\"ufer ring for any $P \in {\rm Spec}(R)$.
\item[(5)]
$R$ is a {\it maximally Pr\"ufer ring}
if $R_M$ is a Pr\"ufer ring for any $M \in {\rm Max}(R)$.
\item[(6)]
$R$ is a {\it Pr\"ufer ring}
if every finitely generated regular ideal is invertible.
\end{enumerate}
In \cite{boynton 2011, glaz 2005, klingler 2015},
the authors proved the implications
(1) $\Rightarrow$ $w.{\rm gl}.\dim(R) \leq 1$ $\Rightarrow$
(2) $\Rightarrow$ (3) $\Rightarrow$ (4) $\Rightarrow$ (5) $\Rightarrow$ (6).

In 1985, the authors showed that
$D[X]_N$ is a Pr\"ufer ring if and only if
$D$ is a strongly Pr\"ufer ring, and
$D[X]_U$ is a Pr\"ufer ring if and only if
$D$ is a strongly Pr\"ufer ring with $\dim(R) \leq 1$ and
if $P \subsetneq Q$ are prime ideals of $R$, then $R_P$ is a field \cite{anderson 1985}.
Also, in 2018, Jarrar and Kabbaj
found the conditions of $R$
under which $R[X]_U$ or $R[X]_N$ becomes such Pr\"ufer-like rings
\cite{jarrar 2018}.
Hence a natural question arises:
When is $R[X]_A$ a Prüfer-like ring?
To answer this question,
we need some facts about the ideal extension to the Anderson ring.

\begin{lemma}\label{number of generator}
Let $R$ be a commutative ring with identity
and let $I$ be an ideal of $R$.
Suppose that $\alpha$ is any cardinal number.
Then $I$ is generated by $\alpha$-elements
if and only if $IR[X]_A$ is generated by $\alpha$-elements.
In particular, $I$ is finitely generated $($respectively, principal$)$
if and only if $IR[X]_A$ is finitely generated $($respectively, principal$)$.
\end{lemma}

\begin{proof}
It is clear that if $I$ is generated by $\alpha$-elements,
then $IR[X]_A$ is generated by $\alpha$-elements.
For the converse,
let $|\Lambda| = \alpha$ and
suppose that $\{f_j \in IR[X] \,|\, j \in \Lambda \}$ is a generating set of $IR[X]_A$.
We claim that $\{f_j(0) \,|\, j \in \Lambda \}$
is a generating set of $I$.
Let $i \in I$.
Then there exist $\alpha_1, \dots, \alpha_n \in \Lambda$,
$g_1, \dots, g_n \in R[X]$
and $h_1, \dots, h_n \in A$ such that
$i = \sum_{j=1}^{n} f_{\alpha_j}\frac{g_j}{h_j}$.
Let $h = h_1 \cdots h_n$ and let $\hat{h}_j = \frac{h}{h_j}$.
Then $ih = \sum_{j=1}^{n}f_{\alpha_j}g_j\hat{h}_j$.
This equation implies that $i=\sum_{j=1}^{n}f_{\alpha_j}(0)g_j(0)$.
Hence $\{f_j(0) \,|\, j \in \Lambda \}$
is a generating set of $I$.
Thus the first argument holds.
The remainder argument follows directly from this result. 
\end{proof}

Readers should note that in Lemma \ref{number of generator},
$\alpha$ may be assumed to represent
the number of minimal generators of $I$ and $IR[X]_A$.

In \cite{anderson 1985},
the authors showed that
if $R$ is an integral domain and
$I$ is an ideal of $R$,
then $IR[X]_U$ is principal if and only if
$I$ is principal.
The following example demonstrates that
in Nagata rings,
Lemma \ref{number of generator} does not hold in general.

\begin{example}
{\rm
Suppose that $D$ is a Pr\"ufer domain which is not a B\'ezout domain.
Then there is a finitely generated ideal of $D$ which is not a principal, say $I$.
Note that $D[X]_N$ is a B\'ezout domain
\cite[Corollary 7]{lim 2014}.
As $ID[X]_N$ is finitely generated,
$ID[X]_N$ is principal.
}
\end{example}

Now, we investigate the invertibility properties of the Anderson ring.

\begin{proposition}\label{locally principal and invertible}
Let $R$ be a commutative ring with identity and let $I$ be an ideal of $R$.
Then the following assertions hold.
\begin{enumerate}
\item[\rm(1)]
$I$ is locally principal if and only if $IR[X]_A$ is locally principal.
\item[\rm(2)]
$IR[X]_A$ is invertible in $R[X]_A$ if and only if
$I$ is finitely generated locally principal with ${\rm ann}(I) = (0)$.
In particular, if $I$ is regular,
then $I$ is invertible in $R$ if and only if
$IR[X]_A$ is invertible in $R[X]_A$.
\end{enumerate}
\end{proposition}

\begin{proof}
(1) Suppose that $I$ is a locally principal ideal of $R$.
Let $\mathfrak{m}$ be a maximal ideal of $R[X]_A$.
By Theorem \ref{maximal cor}(2),
there exists a maximal ideal $M$ of $R$
such that $\mathfrak{m} = (M + XR[X])_A$.
As $IR_M$ is principal, we obtain that
$IR_M[X]_{A_M}$ is principal by Lemma \ref{number of generator}.
Hence $(IR[X]_A)_{\mathfrak{m}}$ is principal
by Lemma \ref{R_P[X]_A_P}(4),
which means that $IR[X]_A$ is locally principal.
For the converse, suppose that $IR[X]_A$ is locally principal.
Let $M$ be a maximal ideal of $R$.
As $(M+XR[X])_A$ is a maximal ideal of $R[X]_A$,
$IR_M[X]_{A_M}=(IR[X]_A)_{(M+XR[X])_A}$ is principal
by Lemma \ref{R_P[X]_A_P}(4).
It follows that $IR_M$ is principal
by Lemma \ref{number of generator}.
Consequently, $I$ is locally principal.

(2) The only if part directly follows from
the fact that $R[X]_N$ is the quotient ring of $R[X]_A$ by $N$
and \cite[Theorem 2.2(5)]{anderson 1985}.
For the converse, suppose that $I$ is
finitely generated locally principal with ${\rm ann}(I) = (0)$.
Then $IR[X]_A$ is finitely generated locally principal by
the assertion (1) and Lemma \ref{number of generator}.
Also, as ${\rm ann}(I) = (0)$,
$IR[X]$ is regular \cite[Chapter I, Exercise 2(iii)]{atiyah book}.
It follows that $IR[X]_A$ is regular.
Thus $IR[X]_A$ is invertible in $R[X]_A$.
The remainder argument is obvious.
\end{proof}

Let $R$ be a commutative ring with identity and
let $T(R)$ be the total quotient ring of $R$.
Recall that an integral domain $D$ is a {\it valuation domain}
if for any nonzero elements $a,b \in D$,
either $(a) \subseteq (b)$ or $(b) \subseteq (a)$.
A well-known fact is that every quasi-local PID is a valuation domain.
Note that $R$ is a semi-hereditary ring if and only if
$T(R)$ is a von Neumann regular ring and
$R_P$ is a valuation domain for all $P \in {\rm Spec}(R)$ \cite[Theorem 2]{endo 1961}.
The next result is the second main result of this section which determines the condition on $R$ under which
$R[X]_A$ becomes Pr\"ufer-like ring.

\begin{theorem}\label{von, Prufer-like}
Let $R$ be a commutative ring with identity.
Then the following assertions are equivalent.
\begin{enumerate}
\item[(1)]
$R$ is a von Neumann regular ring.
\item[(2)]
$R[X]_A$ is a semi-hereditary ring.
\item[(3)]
{\rm w.gl.}$\dim(R[X]_A) \leq 1$.
\item[(4)]
$R[X]_A$ is an arithmetical ring.
\item[(5)]
$R[X]_A$ is a Gaussian ring.
\item[(6)]
$R[X]_A$ is a locally Pr\"ufer ring.
\item[(7)]
$R[X]_A$ is a maximally Pr\"ufer ring.
\item[(8)]
$R[X]_A$ is a Pr\"ufer ring.
\end{enumerate}
In this case, $\dim(R[X]_A) = 1$.
\end{theorem}

\begin{proof}
(1) $\Rightarrow$ (2)
Denote the total quotient ring of $R[X]_A$ by $T(R[X]_A)$.
Suppose that $R$ is a von Neumann regular ring, {\it i.e.},
a zero-dimensional reduced ring.
It is sufficient to show that
$T(R[X]_A)$ is a von Neumann regular ring and
$(R[X]_A)_{\mathfrak{p}}$ is a valuation domain
for all $\mathfrak{p} \in {\rm Spec}(R[X]_A)$ \cite[Theorem 2]{endo 1961}.
As $R[X]_N$ is a overring of $R[X]_A$ and $\dim(R[X]_N) = 0$,
$\dim(T(R[X]_A)) = 0$.
Also, it is clear that $T(R[X]_A)$ is reduced,
so $T(R[X]_A)$ is a von Neumann regular ring.
Now, we claim that 
$(R[X]_A)_{\mathfrak{p}}$ is a valuation domain
for all $\mathfrak{p} \in {\rm Spec}(R[X]_A)$.
As $R$ is zero-dimensional,
every prime ideals of $R[X]_A$ is of the form $MR[X]_A$ or $(M+XR[X])_A$
for some $M \in {\rm Max}(R)$ by Remark \ref{dimension remark}(2).
Note that $R_M$ is a field for all $M \in {\rm Max}(R)$
\cite[Theorem 3.6.16]{wang book}.
This implies that $(R[X]_A)_{MR[X]_A} = R_M[X]_{N_M}$ is a field,
where $N_M = \{f \in R_M[X] \,|\, c(f) = R_M\}$
(cf. \cite[Proposition 5.5.10]{wang book}).
On the other hand, by Corollary \ref{quasi-local}
and Theorem \ref{PIR 2},
$(R[X]_A)_{(M + XR[X])_A} = R_M[X]_{A_M}$ is a quasi-local PID.
Thus our claim holds.
Consequently, $R[X]_A$ is a semi-hereditary ring.

We already mentioned the implications
(2) $\Rightarrow$ (3) $\Rightarrow$
(4) $\Rightarrow$ (5) $\Rightarrow$ (6) $\Rightarrow$ (7) $\Rightarrow$ (8) hold.

(8) $\Rightarrow$ (1)
Suppose that $R[X]_A$ is a Pr\"ufer ring.
It is sufficient to show that $R_M$ is a field for all $M \in {\rm Max}(R)$
\cite[Theorem 3.6.16]{wang book}.
Let $M$ be a maximal ideal of $R$,
$m \in M$ and $I = (m) + XR[X]$.
As $I$ is regular, $IR[X]_A$ is invertible.
Hence $(\frac{m}{1}R_M + \frac{X}{1}R_M[X])_{A_M}$ is principal.
It is easy to check that $\frac{X}{1}R_M[X]_{A_M}$ is not contained in $\frac{m}{1}R_M[X]_{A_M}$,
which follows that $\frac{m}{1}R_M[X]_{A_M} \subseteq \frac{X}{1}R_M[X]_{A_M}$
\cite[Proposition 7.4]{gilmer book}.
This implies that $\frac{m}{1}R_M = (0)$.
Since $m$ is an arbitrary element of $M$,
$MR_M = (0)$.
Thus $R_M$ is a field.
Consequently, $R$ is a von Neumann regular ring.
\end{proof}

Recall that an integral domain $D$ is a {\it B\'ezout domain} if
every finitely generated ideal is principal.
Based on the results obtained so far in this section,
we can derive the following conclusions.

\begin{corollary}\label{sec 3 domain}
Let $D$ be an integral domain.
Then the following assertions are equivalent.
\begin{enumerate}
\item[(1)]
$D$ is a field.
\item[(2)]
$D[X]_A$ is a PID.
\item[(3)]
$D[X]_A$ is a valuation domain.
\item[(4)]
$D[X]_A$ is a B\'ezout domain.
\item[(5)]
$D[X]_A$ is a Pr\"ufer domain.
\end{enumerate}
\end{corollary}

\begin{proof}
The equivalent (1) $\Leftrightarrow$ (2) $\Leftrightarrow$ (5)
directly follow from Theorems \ref{PIR 2} and \ref{von, Prufer-like}.
Suppose that $D$ is a field.
Then $D[X]_A$ is a quasi-local PID by Corollary \ref{quasi-local},
so $D[X]_A$ is a valuation domain.
As every valuation domain is a B\'ezout domain,
$D[X]_A$ is a B\'ezout domain,
and hence $D[X]_A$ is a Pr\"ufer domain.
\end{proof}

\section{Star-operations on $R[X]_A$}\label{sec 4}

In this section,
we investigate star-operations on the Anderson ring.
To help readers better understand this section,
we review some definitions and notation related to star-operations.
In this section, $D$ always denotes an integral domain with quotient field $K$.
Let ${\bf F}(D)$ be the set of nonzero fractional ideals of $D$.
For an $I \in {\bf F}(D)$,
set $I^{-1} := \{a \in K \,|\, aI \subseteq D\}$.
The mapping on ${\bf F}(D)$ defined by $I \mapsto I_v := (I^{-1})^{-1}$
is called the {\it $v$-operation} on $D$;
the mapping on ${\bf F}(D)$ defined by
$I \mapsto I_t :=
\bigcup \{J_v \,|\, J \text{ is a nonzero finitely generated fractional subideal of $I$}\}$
is called the {\it $t$-operation} on $D$.
An ideal $J$ of $D$ is a {\it Glaz–Vasconcelos ideal}
(for short a {\it GV-ideal}),
and denoted by $J \in {\rm GV}(D)$ if
$J$ is finitely generated and $J_v = D$.
For each $I \in {\bf F}(D)$,
the {\it $w$-envelope} of $I$ is the set
$I_w := \{ x \in K \,|\, xJ \subseteq I \text{ for some } J \in {\rm GV}(D)\}$.
The mapping on ${\bf F}(D)$ defined by
$I \mapsto I_w$ is called a {\it $w$-operation} on $D$.
For $*= v,t$ or $w$,
a nonzero fractional ideal $F$ of $D$ is a {\it fractional $*$-ideal} if
$F_* = F$, and
a proper ideal $I$ of $D$ is a {\it maximal $*$-ideal} if
there does not exist a proper $*$-ideal properly containing $I$,
and denoted by $I \in *\mhyphen{\rm Max}(D)$.
The useful facts in this section,
if $D$ is not a field,
then $w\mhyphen{\rm Max}(D) \neq \emptyset$,
$t\mhyphen{\rm Max}(D) = w\mhyphen{\rm Max}(D)$ \cite[Theorem 2.16]{anderson 2000} and
$D = \bigcap_{\mathfrak{m} \in t\mhyphen{\rm Max}(D)} D_{\mathfrak{m}}$
\cite[Proposition 2.9]{kang 1989}.
The readers can refer to \cite{anderson 2000, kang 1989, wang book} for star-operations.

We begin this section with the following lemma.

\begin{lemma}\label{ideal contraction}
Let $R$ be a commutative ring with identity
and let $I, J$ be ideals of $R$.
Then $IR[X]_A \cap R = I$, and hence
$I=J$ if and only if $IR[X]_A = JR[X]_A$.
\end{lemma}

\begin{proof}
Let $r \in IR[X]_A \cap R$.
Then $r = \frac{f}{g}$ for some $f \in IR[X]$ and $g \in A$,
so we obtain the equation $rg=f$.
It follows that $r = rg(0) = f(0) \in I$.
Therefore $IR[X]_A \cap R \subseteq I$.
The reverse containment is obvious.
Thus the first argument holds.
The remainder argument is obvious.
\end{proof}

The next result is a nice tool to investigate star-operations on the Anderson ring.

\begin{proposition}\label{star input}
Let $D$ be an integral domain and
let $I$ be a nonzero fractional ideal of $D$.
Then the following assertions hold.
\begin{enumerate}
\item[\rm(1)]
$(ID[X]_A)^{-1} = I^{-1}D[X]_A$.
\item[\rm(2)]
$(ID[X]_A)_v = I_vD[X]_A$.
\item[\rm(3)]
$(ID[X]_A)_t = I_tD[X]_A$.
\item[\rm(4)]
$(ID[X]_A)_w = I_wD[X]_A$.
\end{enumerate}
\end{proposition}

\begin{proof}
Let $I$ be a nonzero fractional ideal of $D$.
Then there exist a nonzero element $d \in D$ and
a nonzero ideal $J$ of $D$ such that
$J = dI$.
Hence we may assume that $I$ is an integral ideal.

(1) Let $\alpha \in (ID[X]_A)^{-1}$.
Then $\alpha ID[X]_A \subseteq D[X]_A$,
so for any $i \in I$,
$\alpha \in i^{-1}D[X]_A \subseteq K[X]_A$.
This implies that there exist $f \in K[X]$ and $g \in A$ such that
$\alpha = \frac{f}{g}$.
Hence $f \in (ID[X]_A)^{-1}$,
so $fI \subseteq fID[X]_A \subseteq D[X]_A$.
Therefore for each $i$,
there exists $g_i \in A$ such that $ifg_i \in D[X]$,
which means that $ic(f) = ic(fg_i) = c(ifg_i) \subseteq D$.
Hence $c(f)I \subseteq D$, so $c(f) \subseteq I^{-1}$;
that is, $f \in I^{-1}D[X]$.
Thus $\alpha = \frac{f}{g} \in I^{-1}D[X]_A$.
Consequently, $(ID[X]_A)^{-1} \subseteq I^{-1}D[X]_A$.
The reverse containment is obvious.

(2) This result directly follows from the assertion (1).

(3) Let $\alpha \in (ID[X]_A)_t$.
Then there exists a finitely generated ideal $J$ of $D[X]_A$ with $J \subseteq ID[X]_A$
such that $\alpha \in J_v$.
Let $J = (f_1,\dots,f_n)D[X]_A$, where $f_1,\dots,f_n \in ID[X]_A$.
Then $J \subseteq (c(f_1) + \cdots + c(f_n))D[X]_A$.
This implies that $\alpha \in J_v \subseteq (c(f_1) + \cdots + c(f_n))_v D[X]_A
\subseteq I_tD[X]_A$,
where the first containment follow from the assertion (2).
Hence $(ID[X]_A)_t \subseteq I_tD[X]_A$.
For the reverse containment,
let $\frac{f}{g} \in I_tD[X]_A$,
where $f \in I_tD[X]$ and $g \in A$.
Then $c(f) \subseteq I_t$,
so there exists a finitely generated ideal $J$ of $D$ with $J \subseteq I$
such that $c(f) \subseteq J_v$.
Hence $f \in c(f)_vD[X] \subseteq J_vD[X] \subseteq J_vD[X]_A = (JD[X]_A)_v
\subseteq (ID[X]_A)_t$,
where the first equality directly follows from the assertion (2).
Thus $\frac{f}{g} \in (ID[X]_A)_t$.

(4) Let $\frac{f}{g} \in I_wD[X]_A$,
where $f \in I_wD[X]$ and $g \in A$.
Then $c(f) \subseteq I_w$,
so there exists $J \in {\rm GV}(D)$ such that
$c(f)J \subseteq I$.
This implies that $\frac{f}{g}  JD[X]_A \subseteq ID[X]_A$.
By the assertion (1),
$JD[X]_A \in {\rm GV}(D[X]_A)$.
Hence $\frac{f}{g} \in (ID[X]_A)_w$.
For the reverse containment,
let $\frac{f}{g} \in (ID[X]_A)_w$.
As $(ID[X]_A)_w \subseteq D[X]_A$,
we may assume that $f \in D[X]$ and $g \in A$.
This implies that it is sufficient to show that
$f \in I_wD[X]$.
As $f \in (ID[X]_A)_w$,
there exists $J = (\frac{f_1}{g_1},\dots,\frac{f_n}{g_n}) \in {\rm GV}(D[X]_A)$
such that $fJ \subseteq ID[X]_A$.
As $J \subseteq (c(f_1)+ \cdots + c(f_n))D[X]_A$ and $J_v = D[X]_A$,
$(c(f_1)+ \cdots + c(f_n))_vD[X]_A = ((c(f_1)+ \cdots + c(f_n))D[X]_A)_v = D[X]_A$,
where the first equality follows from the assertion (2).
Hence by Lemma \ref{ideal contraction},
$(c(f_1) + \cdots + c(f_n))_v = D$,
which means that $c(f_1) + \cdots + c(f_n) \in {\rm GV}(D)$.
Since for all $1 \leq i \leq n$, $f \frac{f_i}{g_i} \in ID[X]_A$,
for each $1 \leq i \leq n$,
there exist $h_1,\dots,h_n \in A$ such that
$ff_ih_i \in ID[X]$.
Therefore there exists a positive integer $m \in \mathbb{N}$ such that
$c(f_i)^{m+1}c(f) = c(f_i)^m c(ff_i) \subseteq I$ for all $1 \leq i \leq n$
\cite[Theorem 1.7.16]{wang book}.
Hence $c(f)(c(f_1)^{m+1} + \cdots + c(f_n)^{m+1}) \subseteq I$.
As $c(f_1) + \cdots + c(f_n) \in {\rm GV}(D)$,
$c(f_1)^{m+1} + \cdots + c(f_n)^{m+1} \in {\rm GV}(D)$.
This implies that $c(f) \subseteq I_w$,
so $f \in I_wD[X]$.
Thus the equality holds.
\end{proof}

By Proposition \ref{star input},
we can derive several results.

\begin{corollary}\label{*-ideal cor}
Let $D$ be an integral domain and
let $I$ be a nonzero fractional ideal of $D$.
If $* = v,t$ or $w$,
then $I$ is a $*$-ideal if and only if
$ID[X]_A$ is a $*$-ideal.
\end{corollary}

Recall that an $I \in {\bf F}(D)$ is a {$w$-invertible ideal} if
$(II^{-1})_w = D$.

\begin{corollary}\label{w-invertible}
Let $D$ be an integral domain and
let $I$ be a nonzero fractional ideal of $D$.
Then $I$ is $w$-invertible in $D$ if and only if
$ID[X]_A$ is $w$-invertible in $D[X]_A$.
\end{corollary}

\begin{proof}
By the same reason of Proposition \ref{star input},
we may assume that $I$ is an integral ideal of $D$.
Note that $((ID[X]_A)(ID[X]_A)^{-1})_w = (II^{-1})_wD[X]_A$
by Proposition \ref{star input}.
Suppose that $I$ is a $w$-invertible ideal of $D$.
Then $(II^{-1})_w = D$.
This implies that $((ID[X]_A)(ID[X]_A)^{-1})_w = (II^{-1})_wD[X]_A = D[X]_A$.
Hence $ID[X]_A$ is a $w$-invertible ideal of $D[X]_A$.
For the converse,
suppose that $ID[X]_A$ is a $w$-invertible ideal of $D[X]_A$.
Then $(II^{-1})_wD[X]_A = ((ID[X]_A)(ID[X]_A)^{-1})_w = D[X]_A$.
Thus $(II^{-1})_w = D$ by Lemma \ref{ideal contraction}.
Consequently, $I$ is a $w$-invertible ideal of $D$.
\end{proof}

Let $D$ be an integral domain.
In \cite{kang 1989},
the author shows that ${\rm Max}(D[X]_{N_v})
= \{MD[X]_{N_v} \,|\, M \in w \mhyphen {\rm Max}(D)\}$,
where $N_v = \{f \in D[X] \,|\, c(f)_v = D\}$,
and hence $w\mhyphen{\rm Max}(D[X]_{N_v}) = {\rm Max}(D[X]_{N_v})$.
Inspired by this, we intend to characterize the $w$-maximal spectrum
of the Anderson ring.
Recall that a prime ideal $P$ of $D[X]$ is an {\it upper to zero} in $D[X]$ if
$P$ is a nonzero ideal with $P \cap D = (0)$.
It is clear that every upper to zero in $D[X]$ is a height-one prime ideal,
and hence it is a prime $t$-ideal.
The following result is the main theorem of this section.

\begin{theorem}\label{maximal w-ideal}
Let $D$ be an integral domain.
Then $\mathfrak{m}$ is a maximal $w$-ideal of $D[X]_A$ if and only if
$\mathfrak{m}$ is exactly one of the form
\begin{enumerate}
\item[\rm(1)]
$MD[X]_A$ for some maximal $w$-ideal $M$ of $D$, or
\item[\rm(2)]
$\mathfrak{p}D[X]_A$, where $\mathfrak{p} \in w$-${\rm Max}(D[X])$
is an upper to zero in $D[X]$ disjoint from $A$.
\end{enumerate}
\end{theorem}

\begin{proof}
Let $\mathfrak{m}$ be a maximal $w$-ideal of $D[X]_A$.
Then $\mathfrak{p} := \mathfrak{m} \cap D[X]$ is a prime $w$-ideal of $D[X]$ \cite[Lemma 1.2(2)]{KKL 2014}
and $\mathfrak{m} = \mathfrak{p}D[X]_A$.
We first consider the case $\mathfrak{p} \cap D \neq (0)$.
Note that there is a maximal $w$-ideal $\mathfrak{q}$ of $D[X]$ containing $\mathfrak{p}$,
and $\mathfrak{q} = MD[X]$ for some $M \in w\mhyphen{\rm Max}(D)$
\cite[Proposition 2.2]{fontana 1998}.
By Proposition \ref{star input}(4),
$MD[X]_A$ is a $w$-ideal of $D[X]_A$.
As $\mathfrak{m} \subseteq MD[X]_A$ and
$\mathfrak{m}$ is a maximal $w$-ideal of $D[X]_A$,
we obtain $\mathfrak{m} = MD[X]_A$.
Now, suppose that $\mathfrak{p} \cap D = (0)$.
If $\mathfrak{p}$ is not a maximal $w$-ideal of $D[X]$,
then there is a maximal $w$-ideal $\mathfrak{q}$ of $D[X]$
properly containing $\mathfrak{p}$.
Since $\mathfrak{q}$ is not an upper to zero in $D[X]$,
$\mathfrak{q} = MD[X]$ for some $M \in w\mhyphen{\rm Max}(D)$
\cite[Proposition 2.2]{fontana 1998}.
This implies that $MD[X]_A$ is a $w$-ideal of $D[X]_A$
properly containing $\mathfrak{m}$.
This contradicts to the fact that $\mathfrak{m}$ is a maximal $w$-ideal of $D[X]_A$.
Hence $\mathfrak{p}$ is a maximal $w$-ideal of $D[X]$.
Now, we claim that every ideal of the forms
(1) and (2) is a maximal $w$-ideal of $D[X]_A$.
Let $M$ be a maximal $w$-ideal of $D$.
Then $MD[X]_A$ is a $w$-ideal of $D[X]_A$ by Proposition \ref{star input}(4).
Hence there is a maximal $w$-ideal $\mathfrak{m}$ of $D[X]_A$ containing $MD[X]_A$.
As $\mathfrak{m} \cap D \neq (0)$,
there is a maximal $w$-ideal $M_1$ of $D$ such that
$\mathfrak{m} = M_1D[X]_A$ by the above argument.
Since $M \subseteq M_1$ and $M$ is a maximal $w$-ideal of $D$,
$M = M_1$,
which shows that $MD[X]_A$ is a maximal $w$-ideal of $D[X]_A$.
Next, assume that $\mathfrak{p} \in w\mhyphen{\rm Max}(D[X])$
is an upper zero in $D[X]$ disjoint from $A$.
Then $\mathfrak{p}D[X]_A$ is a height-one prime ideal,
so it is a $w$-ideal.
Thus $\mathfrak{p}D[X]_A$ is a maximal $w$-ideal of $D[X]_A$ by Proposition \ref{star input}(4).
Consequently, every maximal $w$-ideal of $D[X]_A$ is of the form
(1) or (2).
\end{proof}

We can derive many facts from Theorem \ref{maximal w-ideal}.
From now on, we investigate some properties related to $w$-maximal spectrum of the Anderson ring.

First, we examine the $w$-(Krull) dimension of the Anderson ring.
Let $D$ be an integral domain.
Note that every prime ideal containing a maximal $w$-ideal is a $w$-ideal.
According to this fact,
Wang defined the {\it $w$-$($Krull$)$ dimension} of
an integral domain $D$, denoted by
$w \mhyphen\dim(D)$, as
$w \mhyphen\dim (D) =
\sup\{{\rm ht}(M) \,|\, M \in w \mhyphen {\rm Max}(D)\}$
in \cite{wang 1999-2}.
A well-known fact is that
$n \leq w\mhyphen \dim(D[X]) \leq 2n$ when $w\mhyphen \dim(D) = n$.
Using only the definition of the $w$-dimensions and the fact above,
we can easily derive the following result.

\begin{corollary}
Let $D$ be an integral domain.
Then $w\mhyphen\dim(D[X]) = w\mhyphen\dim(D[X]_A)$,
and hence $n \leq w\mhyphen\dim(D[X]_A) \leq 2n$ if
$w\mhyphen\dim(D) = n$.
\end{corollary}

Let $R$ be a commutative ring with identity and
let $D$ be an integral domain.
It is a well-known fact that
every prime ideal is contained in a maximal ideal.
If such maximal ideal is unique,
the ring $R$ is called a {\it pm-ring}.
Similarly, it is easy to show that
every prime $w$-ideal is contained in a maximal $w$-ideal.
If such a maximal $w$-ideal is unique,
then the integral domain $D$ is called a {\it $w$-$pm$-domain}.
Recall that $D$ is a {\it UMT-domain} if
every upper to zero in $D[X]$ is a maximal $w$-ideal.

\begin{corollary}\label{w-pm-domain}
Let $D$ be an integral domain.
If $D[X]_A$ is a $w$-$pm$-domain,
then so is $D$.
Moreover, the converse holds when $D$ is a UMT-domain.
\end{corollary}

\begin{proof}
Suppose that $D[X]_A$ is a $w$-$pm$-domain,
and suppose to the contrary that $D$ is not a $w$-$pm$-domain.
Then there is a prime $w$-ideal $P$ of $D$
which is contained in two distinct maximal $w$-ideals of $D$,
say $M_1$ and $M_2$.
Hence $PD[X]_A$ is a nonzero prime ideal of $D[X]_A$ which is contained in
$M_1D[X]_A$ and $M_2D[X]_A$.
Since $M_1D[X]_A$ and $M_2D[X]_A$ are distinct maximal $w$-ideal of $D[X]_A$
by Lemma \ref{ideal contraction} and Theorem \ref{maximal w-ideal},
our assumption is false since $D[X]_A$ is a $w$-$pm$-domain.
Thus $D$ is a $w$-$pm$-domain.

For the remainder argument,
suppose that $D$ is a $w$-$pm$-domain, and
suppose to the contrary that $D[X]_A$ is not a $w$-$pm$-domain.
Then there is a prime $w$-ideal $\mathfrak{p}$ of $D[X]_A$
which is contained in two distinct maximal $w$-ideals of $D[X]_A$,
say $\mathfrak{m}_1$ and $\mathfrak{m}_2$.
If $\mathfrak{p} \cap D = (0)$,
then $\mathfrak{p} = \mathfrak{q}D[X]_A$,
where $\mathfrak{q}$ is an upper to zero in $D[X]$ disjoint from $A$.
Since $D$ is a UMT-domain,
$\mathfrak{q}$ is a maximal $w$-ideal of $D[X]$.
Hence $\mathfrak{p}$ is a maximal $w$-ideal of $D[X]_A$
by Theorem \ref{maximal w-ideal}.
This contradicts to our assumption.
Hence $\mathfrak{p} \cap D \neq (0)$.
This implies that
there exist maximal $w$-ideal $M_1$ and $M_2$ of $D$ such that
$\mathfrak{m}_1 = M_1 D[X]_A$ and
$\mathfrak{m}_2 = M_2 D[X]_A$ by Theorem \ref{maximal w-ideal}.
This implies that $\mathfrak{p} \cap D$ is contained in both $M_1$ and $M_2$,
a contradiction since $M_1 \neq M_2$ by 
Lemma \ref{ideal contraction}.
Hence $D[X]_A$ is a $w$-$pm$-domain.
\end{proof}

Let $D$ be an integral domain and
let $M$ be a $D$-module.
An ideal $I$ of $D$ is a {\it trace ideal} if $I = II^{-1}$.
In \cite{fontana 1987, heinzer 1988, lucas 1996, lucas 1999},
the authors have characterized integral domains using the trace ideal as follows:
\begin{enumerate}
\item[\rm(1)]
$D$ is a {\it TP domain} if every trace ideal of $D$ is prime.
\item[\rm(2)]
$D$ is an {\it RTP domain} if every trace ideal of $D$ is radical.
\item[\rm(3)]
$D$ is a {\it TPP domain} if the trace ideal of noninvertible primary ideal is prime.
\item[\rm(4)]
$D$ is an {\it LTP domain} if for each trace ideal $I$ of $D$ and
all minimal prime ideal $P$ of $I$, $ID_P = PD_P$.
\end{enumerate}
In \cite{lucas 1996, lucas 1999},
the authors proved the implications (1) $\Rightarrow$ (2) $\Rightarrow$ (3) $\Rightarrow$ (4).

In \cite{lucas 2020}, the authors investigated the trace properties of
Nagata rings and Serre's conjecture rings.
Note that if $D$ is a PID,
then $D$ has these trace properties.
This implies that if $D$ is a field,
then $D[X]_A$ has such trace properties
by Corollary \ref{sec 3 domain}.
Note that if $D$ is an LTP domain,
then every maximal ideal of $D$ is a $t$-ideal
\cite[Theorem 5(a)]{lucas 1999}.
Now, recall that $w\mhyphen{\rm Max}(D) = t\mhyphen{\rm Max}(D)$,
which means that if $D$ is not a field,
then every maximal ideal of $D[X]_A$ is not a $t$-ideal
by Theorems \ref{maximal cor}(2) and \ref{maximal w-ideal}.
This fact directly implies the following result,
which is one of the results that can be obtained from the fact that
${\rm Max}(D[X]_A) \neq w \mhyphen {\rm Max}(D[X]_A)$.

\begin{corollary}
Let $D$ be an integral domain.
Then the following assertions are equivalent.
\begin{enumerate}
\item[\rm(1)]
$D$ is a field.
\item[\rm(2)]
$D[X]_A$ is a TP domain.
\item[\rm(3)]
$D[X]_A$ is an RTP domain.
\item[\rm(4)]
$D[X]_A$ is a TPP domain.
\item[\rm(5)]
$D[X]_A$ is an LTP domain.
\end{enumerate}
\end{corollary}

Let $D$ be an integral domain.
Recall that $D$ is an {\it H-domain} if
for any ideal $I$ of $D$ with $I^{-1} = D$,
there exists $J \in {\rm GV}(D)$ such that $J \subseteq I$.
Now, we investigate the condition on $D$
under which the Anderson ring become
H-domains when $D$ is integrally closed.

\begin{proposition}\label{H-domain}
Let $D$ be an integrally closed domain.
Then $D$ is an H-domain if and only if
$D[X]_A$ is an H-domain.
\end{proposition}

\begin{proof}
Suppose that $D[X]_A$ is an H-domain.
It is sufficient to show that every maximal $w$-ideal of $D$
is a $v$-ideal of $D$ \cite[Theorem 7.4.2]{wang book}.
Let $M$ be a maximal $w$-ideal of $D$.
Then $MD[X]_A$ is a maximal $w$-ideal of $D[X]_A$ by Proposition \ref{star input}(4).
Hence $M[X]_A$ is a $v$-ideal,
which shows that $M$ is a $v$-ideal of $D$ by Corollary \ref{*-ideal cor}.
For the converse,
suppose that $D$ is an H-domain.
Then every maximal $w$-ideal of $D$ is a $v$-ideal \cite[Theorem 7.4.2]{wang book}.
Let $\mathfrak{m}$ be a maximal $w$-ideal of $D[X]_A$.
By Theorem \ref{maximal w-ideal},
either $\mathfrak{m} = MD[X]_A$ or $\mathfrak{m} = \mathfrak{p}D[X]_A$,
where $\mathfrak{p} \in w\mhyphen{\rm Max}(D[X])$ is an upper to zero in $D[X]$
disjoint from $A$ and $M \in w\mhyphen{\rm Max}(D)$.
Since $M$ is a $v$-ideal of $D$,
$MD[X]_A$ is a $v$-ideal of $D[X]_A$ by Corollary \ref{*-ideal cor}.
Note that $\mathfrak{p}$ is a $w$-invertible ideal of $D[X]$
\cite[Theorem 7.3.14]{wang book},
so $\mathfrak{p}$ is a $v$-ideal of $D[X]$
\cite[Theorem 7.2.14(2)]{wang book}.
Also, since $D$ is integrally closed,
there exist $f \in K[X]$ and a fractional ideal $I$ of $D$ 
such that $\mathfrak{p} = fID[X]$.
Hence $\mathfrak{m} = \mathfrak{p}D[X]_A$ is a $v$-ideal of $D[X]$
by Proposition \ref{star input}(2).
\end{proof}

Recall that an integral domain $D$ has {\it finite $w$-character} if
any nonzero nonunit element of $D$ is contained in
only a finite number of maximal $w$-ideals of $D$.
Recall that $D$ has finite character does not necessarily imply that
$D[X]_A$ has finite character by Proposition \ref{finite character}.
However, the next result shows that
if $D$ has finite $w$-character,
then $D[X]_A$ has finite $w$-character.

\begin{proposition}\label{w-finite character}
Let $D$ be an integral domain.
Then $D$ has finite $w$-character if and only if
$D[X]_A$ has finite $w$-character.
\end{proposition}

\begin{proof}
Suppose that $D[X]_A$ has finite $w$-character.
Let $a$ be a nonzero nonunit element of $D$.
Then $a$ is contained in only a finite number of maximal $w$-ideals of $D[X]_A$,
say $\mathfrak{m}_1,\dots,\mathfrak{m}_n$.
By Theorem \ref{maximal w-ideal},
for each $1 \leq i \leq n$,
there exists a maximal $w$-ideal $M_i$ of $D$ such that
$\mathfrak{m}_i =  M_iD[X]_A$ since $a$ is not contained in any upper to zero in $D[X]$.
Hence $a \in M_i$ by Lemma \ref{ideal contraction}.
Suppose to the contrary that
there is a maximal $w$-ideal $M$ of $D$ distinct to $M_1,\dots M_n$ such that
$a \in M$.
Then $a \in MD[X]_A$.
This contradicts to the fact that $MD[X]_A$ is a maximal $w$-ideal of $D[X]_A$
distinct to $\mathfrak{m}_1,\dots,\mathfrak{m}_n$.
This follows that $a$ is contained in only a finite number of maximal $w$-ideals of $D$.
For the converse, suppose that $D$ has finite $w$-character.
Let $\frac{f}{g}$ be a nonzero nonunit element of $D[X]_A$.
Suppose to the contrary that $\frac{f}{g}$ is contained in
an infinite number of maximal $w$-ideals of $D[X]_A$.
Then $f$ is contained in an infinite number of maximal $w$-ideals of $D[X]$.
If there exist an infinite number of maximal $w$-ideals of $D[X]$ containing $f$
which are not an upper to zero in $D[X]$,
then the coefficient of least degree term of $f$
is contained in an infinite number of maximal $w$-ideals of $D$.
This contradicts to the fact that $D$ has finite $w$-character.
Hence $f$ is contained in an infinite number of upper to zero maximal $w$-ideals of $D[X]$.
Let $\{Q_{\alpha} \,|\, \alpha \in \Lambda\}$ be the set of
upper to zero maximal $w$-ideal of $D[X]$ containing $f$.
Note that for each $\alpha \in \Lambda$,
there exists irreducible polynomial $f_{\alpha} \in K[X]$ such that
$Q_{\alpha} = f_{\alpha}K[X]\cap D[X]$.
This implies that $f$ has an infinite number of irreducible polynomial factors in $K[X]$,
a contradiction.
Thus $f$ is contained in only a finite number of maximal $w$-ideals of $D[X]$.
Consequently, $D[X]_A$ has finite $w$-character.
\end{proof}

Let $D$ be an integral domain.
Then $D$ is a {\it weakly Matlis domain} if
$D$ is a $w$-$pm$-domain which has finite $w$-character.
According to Corollary \ref{w-pm-domain} and
Proposition \ref{w-finite character},
we can directly obtain the following result.

\begin{corollary}
Let $D$ be an integral domain.
If $D[X]_A$ is a weakly Matlis domain,
then so is $D$.
Moreover, the converse holds
when $D$ is a UMT-domain.
\end{corollary}

An integral domain $D$ is a {\it $w$-almost Dedekind domain} if
$D_M$ is a Dedekind domain for all $M \in w\mhyphen{\rm Max}(D)$.

\begin{proposition}\label{w-almost Dedekind}
Let $D$ be an integral domain.
Then $D$ is a $w$-almost Dedekind domain if and only if
$D[X]_A$ is a $w$-almost Dedekind domain.
\end{proposition}

\begin{proof}
Suppose that $D$ is a $w$-almost Dedekind domain.
Let $\mathfrak{m}$ be a maximal $w$-ideal of $D[X]_A$.
By Theorem \ref{maximal w-ideal},
either $\mathfrak{m} = MD[X]_A$ or $\mathfrak{m} = \mathfrak{p}D[X]_A$,
where $M \in w \mhyphen{\rm Max}(D)$ and
$\mathfrak{p} \in w \mhyphen{\rm Max}(D[X])$
is an upper to zero in $D[X]$ disjoint from $A$.
This implies that $(D[X]_A)_{\mathfrak{m}}$ is equal to either
$D_M[X]_{N_M}$ or $D[X]_{\mathfrak{p}}$,
where $N_M = \{f \in D_M[X] \,|\, c(f) = D_M\}$.
Since $D_M$ is a Dedekind domain,
$D_M[X]_{N_M}$ is also a Dedekind domain \cite[Theorem 5.4(1)]{anderson 1985}.
Also, note that $(D[X]_A)_{\mathfrak{p}}$ is a DVR \cite[Exercise 5.31]{wang book},
so it is a Dedekind domain since $(D[X]_A)_{\mathfrak{p}}$ is quasi-local.
This follows that $(D[X]_A)_{\mathfrak{m}}$ is a Dedekind domain
for all $\mathfrak{m} \in w \mhyphen{\rm Max}(D[X]_A)$.
Hence $D[X]_A$ is a $w$-almost Dedekind domain.
For the converse,
suppose that $D[X]_A$ is a $w$-almost Dedekind domain.
Let $M$ be a maximal $w$-ideal of $D[X]_A$.
Then $MD[X]_A$ is a maximal $w$-ideal of $D[X]_A$
by Theorem \ref{maximal w-ideal}.
Hence $D_M[X]_{N_M} = (D[X]_A)_{MD[X]_A}$ is a Dedekind domain,
which shows that $D_M$ is a Dedekind domain \cite[Theorem 5.4(1)]{anderson 1985}.
Thus $D$ is a $w$-almost Dedekind domain.
\end{proof}

Similar to Proposition \ref{w-almost Dedekind},
we can naturally consider the $w$-local properties of the Anderson ring.
In fact, we can derive many $w$-local properties of the Anderson ring.
An integral domain $D$ is a {\it $w$-locally Noetherian domain} if
$D_M$ is a Noetherian domain for all $M \in w\mhyphen{\rm Max}(D)$.

\begin{remark}\label{w-local properties}
{\rm
Let $D$ be an integral domain.
Note that $D$ is a $w$-locally Noetherian domain
if and only if $D[X]_N$ is a $w$-locally Noetherian domain.
Also, if $D$ is a DVR,
then $D$ is a Noetherian domain.
By the similar to the proof of Proposition \ref{w-almost Dedekind},
we obtain
\begin{enumerate}
\item[\rm(1)]
$D$ is a $w$-locally Noetherian domain if and only if
$D[X]_A$ is a $w$-locally Noetherian domain.
\end{enumerate}
Similarly, if every DVR has a property (P),
and $D$ has a property (P) if and only if $D[X]_N$ has a property (P),
then we obtain
\begin{enumerate}
\item[\rm(2)]
$D_M$ has a property (P) for all $M \in w\mhyphen{\rm Max}(D)$
if and only if
$(D[X]_A)_{\mathfrak{m}}$ has a property (P)
for all $\mathfrak{m} \in w\mhyphen{\rm Max}(D[X]_A)$.
\end{enumerate}
}
\end{remark}

Let $D$ be an integral domain.
An ideal $I$ of $D$ is of {\it $w$-finite type} if
there exists a finitely generated subideal $J$ of $I$ such that
$I_w = J_w$.
Recall that $D$ is a {\it strong Mori domain} if
every nonzero ideal of $D$ is of $w$-finite type.
A well-known fact of strong Mori domains is that
$D$ is a strong Mori domain if and only if
$D$ is a $w$-locally Noetherian domain and
has finite $w$-character \cite[Theorem 1.9]{wang 1999}.
By Proposition \ref{w-finite character} and Remark \ref{w-local properties}(1), we have

\begin{corollary}{\rm (cf. \cite[Theorem 2.2]{chang 2005})}
Let $D$ be an integral domain.
Then $D$ is a strong Mori domain if and only if
$D[X]_A$ is a strong Mori domain.
\end{corollary}


\pagestyle{plain}
\bibliographystyle{amsplain}

\end{document}